\documentclass[12pt]{article}
\usepackage{amsthm,amsmath,amssymb}
\theoremstyle{plain}
\newtheorem{theorem}{Theorem}
\newtheorem{proposition}[theorem]{Proposition}
\newtheorem{lemma}[theorem]{Lemma}
\newtheorem{corollary}[theorem]{Corollary}
\theoremstyle{definition}
\newtheorem{remark}[theorem]{Remark}
\newcommand{\abs}[1]{\lvert#1\rvert}
\newcommand{\FF}{\mathbb F}
\DeclareMathOperator{\gL}{\Gamma L}
\DeclareMathOperator{\GL}{GL}
\DeclareMathOperator{\PGL}{PGL}
\usepackage{xcolor} 
\usepackage{url}
\usepackage[colorlinks,linkcolor=blue,citecolor=purple]{hyperref}
\begin{document}
\title{A note on Galois groups of linearized polynomials}
\author{Peter M\"uller\\[1mm]
  Institute of Mathematics, University of W\"urzburg\\
  \texttt{peter.mueller@uni-wuerzburg.de}}
\maketitle
\begin{abstract}
  Let $L(X)$ be a monic $q$-linearized polynomial over $\FF_q$ of
  degree $q^n$, where $n$ is an odd prime. In \cite{gow_mcguire_glnq},
  Gow and McGuire showed that the Galois group of $L(X)/X-t$ over the
  field of rational functions $\FF_q(t)$ is $\GL_n(q)$ unless
  $L(X)=X^{q^n}$. The case of even $q$ remained open, but it was
  conjectured that the result holds too and partial results were
  given. In this note we settle this conjecture. In fact we use
  Hensel's Lemma to give a unified proof for all prime powers $q$.
\end{abstract}
\section{Introduction}
In this note we extend the main result \cite[Theorem
2]{gow_mcguire_glnq} from odd prime powers $q$ to all prime powers. A
$q$-linearized polynomial over $\FF_q$ of $q$-degree $n$ is defined as
$L(X)=\sum_{i=0}^na_iX^{q^i}$ where $a_i\in\FF_q$ and $a_n\ne0$.
\begin{theorem}\label{t:main}
  Let $q$ be a prime power and $n$ be an odd prime. Let $L(X)$ be a
  monic $q$-linearized polynomial over $\FF_q$ of $q$-degree $n$. Then
  the Galois group of $L(X)/X-t$ over $\FF_q(t)$ is $\GL_n(q)$ unless
  $L(X)=X^{q^n}$.
\end{theorem}
Using Hensel's Lemma we prove the following divisibility result for
the order of the Galois group. Note that $n$ need neither be odd, nor
be a prime.
\begin{proposition}\label{p:main}
  Let $q$ be a prime power and $n$ be a positive integer. Set
  $L(X)=\sum_{i=m}^na_iX^{q^i}$ where $a_i\in\FF_q$, $a_m, a_n\ne0$
  and $1\le m\le n-1$. Then $q^m\cdot(q^m-1)\cdot(q^n-1)$ divides the
  order of the Galois group of $L(X)/X-t$ over $\FF_q(t)$.
\end{proposition}
We show how Theorem \ref{t:main} follows from Proposition
\ref{p:main}. Upon replacing $t$ with $t-a_0$ one may assume that
$a_0=0$. Let $G$ be the Galois group of $L(X)/X-t$ over $\FF_q(t)$. It
is well known and easy to prove that $G\le\GL_n(q)$. Furthermore,
ramification at $t=\infty$ tells that $G$ contains a cyclic regular
subgroup of order $q^n-1$. (See e.g.\ \cite[Lemma 3.3]{Turnwald:Schur}
for a short proof.) A group theoretic classification result (see
Remark \ref{r:singer}) shows that $G=\GL_n(q)$ or
$G\le\gL_1(q^n)$. The latter group has order $n\cdot(q^n-1)$. Suppose
that $G\ne\GL_n(q)$. Proposition \ref{p:main} yields
$q^m\cdot(q^m-1)\cdot(q^n-1)\mid n\cdot(q^n-1)$. This implies
$q^m\mid n$, hence $m=1$ and $q=n$ because $n$ is a prime. The
contradiction $n-1\mid1$ shows that there is no $m$ with
$1\le m\le n-1$ such that $a_m\ne0$, hence we have the excluded case
$L(X)=X^{q^n}$.
\section{Proof of Proposition \ref{p:main}}
In the following we let $y$ be a variable over the field $K$ and $X$
be a variable over the field of rational functions $K(y)$.

The proof of the following lemma uses Hensel's Lemma and the
Eisenstein irreducibility criterion for polynomials with coefficients
in a power series ring over a field.
\begin{lemma}\label{l:main}
  Let $f$ and $g$ be non-constant polynomials over $K$. Suppose that
  \begin{itemize}
  \item[(i)] $f$ has a root of multiplicity $a$,
  \item[(ii)] $g$ has a root of multiplicity $b$,
  \item[(iii)] $a$ and $b$ are relatively prime, and
  \item[(iv)] $f(X)-g(y)$ is separable over $K(y)$.
  \end{itemize}
  Then $a$ divides the order of the Galois group of $f(X)-g(y)$ over
  $K(y)$.
\end{lemma}
\begin{proof}
  Let $E$ be a common splitting field of $f$ and $g$ over $K$, and
  $\alpha$ and $\beta$ be the roots of $f$ and $g$ of multiplicity $a$
  and $b$, respectively. As $a$ and $b$ are relatively prime, there
  are positive integers $r$ and $s$ such that $sb-ra=1$. Choose $z$
  with $z^s=y$. Let $E[[z]]$ be the ring of formal power series in $z$
  over $E$ and $E((z))$ be its quotient field, which is the field of
  formal Laurent series in $z$.

  By definition, $K(y)$ is a subfield of $E((z))$, so the Galois group
  of $f(X)-g(y)$ over $E((z))$ is a subgroup of the Galois group of
  $f(X)-g(y)$ over $K(y)$.

  Thus we are done once we know that $f(X)-g(y)$ has an irreducible
  factor over $E((z))$ of degree $a$. Upon replacing $X$ and $y$ with
  $X+\alpha$ and $y+\beta$ we may and do assume that $\alpha=\beta=0$.

  Thus $f(X)=X^au(X)$ and $g(y)=y^bv(y)=z^{sb}v(z^s)$ with
  $u(0)\ne0\ne v(0)$. The degrees of the irreducible factors of
  $f(X)-g(y)$ over $E((z))$ don't change if we replace $X$ with
  $Xz^r$. Recall that $sb-ra=1$, so we are concerned with the
  factorization of\[ H(X)=X^au(Xz^r)-zv(z^s).
  \]
  As $X^a$ and $u(Xz^r)$ are relatively prime, Hensel's Lemma provides
  a factorization $H(X)=A(X)U(X)$ over $E[[z]]$ such that
  $A(X)\equiv X^a\pmod{z}$. 
  Furthermore, $A(0)U(0)=H(0)=-zv(z^s)$, so $z^2$ does not divide
  $A(0)$. Therefore $A(X)$ is an Eisenstein polynomial of degree $a$
  with respect to $z$, and we are done.
\end{proof}
\begin{remark}
  A variant of the proof would use standard facts about Newton
  polygons. Assume as above that $f(X)=Z^au(X)$ and
  $g(y)=y^bv(y)$. Let $\nu$ be a valuation of the splitting field of
  $f(X)-g(y)$ over $E((y))$ ($E$ as above) which extends the $y$-adic
  valuation of $E((y))$.

  The Newton polygon of $f(X)-g(y)$ with respect to $\nu$ consists of
  two line segments. The first segment connects $(0, b)$ with
  $(a, 0)$, and the second segment connects $(a, 0)$ with $(n, 0)$.

  From known properties of the Newton polygon, we get that $f(X)-g(y)$
  has exactly $a$ roots $\gamma$ with $\nu(\gamma)=\frac{b}{a}$. Let
  $\Gamma$ be the set of these roots. Then
  $\prod_{\gamma\in\Gamma}(X-\gamma)$ has coefficients in $E((y))$ and
  is irreducible over $E((y))$.
\end{remark}
\begin{corollary}
  Let $f\in K[X]$ be a polynomial of degree $n$ for which two of its
  roots (in some extension of $K$) have relatively prime
  multiplicities $a$ and $b$. Furthermore, assume that $f(X)-t$ is
  separable over the field $K(t)$ of rational functions in $t$. Then
  $a\cdot b\cdot n$ divides the order of the Galois group of $f(X)-t$
  over $K(t)$.
\end{corollary}
\begin{proof}
  Let $G$ be the Galois group of $f(X)-t$ over $K(t)$. Define $y$ in
  an extension of $K(t)$ by $f(y)=t$. Then $y$ is a root of
  $f(X)-t$. So the Galois group $H$ of $f(X)-t=f(X)-f(y)$ over $K(y)$
  is the stabilizer in $G$ of the root $y$. As $f(X)-t$ is irreducible
  over $K(t)$ we have $[G:H]=\deg (f(X)-t)=n$. Thus we need to show
  that $a\cdot b$ divides $\abs{H}$.

  But this follows from Lemma \ref{l:main} with $f=g$. First we use it
  with a root of multiplicity $a$ of $f(X)$ and a root of multiplicity
  $b$ of $f(y)$, and secondly with a root of multiplicity $b$ of
  $f(X)$ and a root of multiplicity $a$ of $f(y)$.
\end{proof}
In order to prove Proposition \ref{p:main}, we are left to show that
$L(X)/X$ has roots of multiplicities $q^m-1$ and $q^m$. We compute
\begin{align*}
  L(X)/X%
  &= \sum_{i=m}^na_iX^{q^i-1}\\
  &= X^{q^m-1}\sum_{i=m}^na_iX^{q^i-q^m}\\
  &= X^{q^m-1}(\sum_{i=m}^na_iX^{q^{i-m}-1})^{q^m}.
\end{align*}
Set $h(X)=\sum_{i=m}^na_iX^{q^{i-m}-1}$. Note that
$a_m-Xh'(X)=h(X)$. As $a_m\ne0$, we see that $h$ is separable and
$h(0)\ne0$. From that we obtain the assertion about the multiplicities
of the roots of $L(X)/X$.
\begin{remark}\label{r:singer}
  The main tool from group theory in the proof of Theorem \ref{t:main}
  is \cite[Theorem 7]{gow_mcguire_glnq}: If a subgroup $G$ of
  $\GL_n(q)$ contains a Singer cycle (an element which permutes the
  nonzero elements of $\FF_q^n$ cyclically), then
  $\GL_{n/d}(q^d)\le G\le\gL_{n/d}(q^d)$ for some divisor $d$ of $n$.

  For this result Gow and McGuire refer to several much stronger results
  about transitive linear groups. In particular, these sources depend
  on the classification of the finite simple groups.

  In \cite{Kantor:Singer}, Kantor gives a short proof of this result
  without using the classification of the finite simple
  groups. His proof is based on the classification-free paper
  \cite{CamKan:2tranproj} where Cameron and Kantor determine the doubly
  transitive subgroups of $\PGL_n(q)$, which later turned out to
  contain mistakes. These were fixed by Kantor in the revision
  \cite{1806.02203}.
\end{remark}

\end{document}